\documentclass[a4paper]{amsart}
\usepackage[english]{babel}
\usepackage{amsmath,amssymb,amsthm}

\usepackage[pdftex]{color}

\usepackage{graphicx}

\usepackage[bookmarks=true,hyperindex,pdftex,colorlinks,citecolor=blue]
{hyperref}

\hypersetup{pdftitle={CKuniformlyconvex},
pdfauthor={Sun Kwang Kim and  Han Ju Lee},pdfpagemode=UseOutlines}

\theoremstyle{plain}
\newtheorem{theorem}{Theorem}[section]

\newtheorem{proposition}[theorem]{Proposition}
\newtheorem{lemma}[theorem]{Lemma}

\theoremstyle{definition}

\newtheorem{definition}[theorem]{Definition}

 \DeclareMathOperator{\re}{Re\,}

\newcommand{\eps}{\varepsilon}

\renewcommand{\leq}{\leqslant}
\renewcommand{\geq}{\geqslant}
\renewcommand{\leq}{\leqslant}
\renewcommand{\ge}{\geqslant}

\begin{document}

\title[The Bishop-Phelps-Bollob\'as property for operators from $\mathcal{C}(K)$]{The Bishop-Phelps-Bollob\'as property for operators from $\mathcal{C}(K)$  to uniformly convex spaces}

\author[Kim]{Sun Kwang Kim}
\address[Kim]{Department of Mathematics, Kyonggi University, Suwon 443-760, Republic of Korea}
\email{\texttt{sunkwang@kgu.ac.kr}}

\author[Lee]{Han Ju Lee}
\address[Lee]{Department of Mathematics Education,
Dongguk University - Seoul, 100-715 Seoul, Republic of Korea}
\email{\texttt{hanjulee@dongguk.edu}}

\thanks{The second named author partially supported by Basic Science Research Program through the National Research Foundation of Korea(NRF) funded by the Ministry of Education, Science and Technology (NRF-2012R1A1A1006869).}



\begin{abstract} We show that the pair $(C(K),X)$ has the Bishop-Phelps-Bolloba\'as property for operators if $K$ is a compact Hausdorff space and $X$ is a uniformly convex space.
\end{abstract}

\date{June 26th, 2014}

\maketitle

\section{Introduction}
In this paper, we deal with strengthening of the famous Bishop-Phelps theorem. In 1961, Bishop and Phelps \cite{BP} showed that the set of all norm attaining functionals on a Banach space $X$ is dense in its dual space $X^*$ which is now called Bishop-Phelps theorem. This theorem has been extended to operators between Banach spaces $X$ and $Y$. In general, the set of norm attaining operators $\mathcal{NA}(X,Y)$ is not dense in the space of linear operators $\mathcal{L}(X,Y)$. However, it is true for some pair of Banach spaces $(X,Y)$. One of very well-known examples is the pair of every reflexive Banach space $X$ and every Banach space $Y$, which was shown by Lindenstrauss \cite{L}. After that, this was generalized by Bourgain to Banach space $X$ with Radon-Nikod\'ym property \cite {Bou}, and also there have been many efforts to find other positive examples \cite{CKlush, CK2, CLS1, FP, JoWo, S, U}.

Meanwhile, Bollob\'as sharpened Bishop-Phelps theorem as follows. From now on, the unit ball and the unit sphere of a Banach space $X$ will be denoted by $B_X$ and $S_X$, respectively.

\begin{theorem}(\cite{Bol})
For an arbitrary $\epsilon>0$, if $x^*\in S_{X^*}$ satisfies $|1-x^*(x)|<{{\epsilon^2}\over{4}}$ for $x\in B_X$, then there are both $y\in S_X$ and $y^*\in S_{X^*}$ such that $y^*(y)=1$, $\|y-x\|<\epsilon$ and $\|y^*-x^*\|<\epsilon$.
\end{theorem}

 This Bishop-Phelps-Bollob\'as theorem shows that if a functional almost attains its norm at a point, then it is possible to approximate simultaneously both the functional and the point by norm attaining functionals and their corresponding norm attaining points. Clearly, Bishop-Phelps-Bollob\'as theorem implies Bishop-Phelps theorem.

 Similarly to the case of Bishop-Phelps theorem, Acosta, Aron, Garc\'ia and Maestre \cite{AAGM2} started to extend this theorem to bounded linear operators between Banach spaces and introduced the new notion  \emph{Bishop-Phelps-Bollob\'as property}.

\begin{definition} (\cite[Definition 1.1]{AAGM2}) Let $X$ and $Y$ be Banach spaces. We say that the pair $(X, Y)$ has the Bishop-Phelps-Bollob\'as property for operators (\emph{BPBp}) if, given $\epsilon >0$,  there exists $\eta(\epsilon)>0$  such that if there exist both $T\in S_{\mathcal{L}(X,Y)}$ and $x_0 \in S_X$ satisfying $\|Tx_0\|>1-\eta (\epsilon)$, then there exist both an operator $S\in S_{\mathcal{L}(X,Y)}$ and $u_0\in S_X$ such that
$$\|Su_0\|=1, \|x_0 - u_0\|<\epsilon ~\mbox{and}~ \|T-S\|<\epsilon.$$
\end{definition}

 Acosta et al. showed \cite{AAGM2} that the pair $(X,Y)$ has the \emph{BPBp} for finite dimensional Banach spaces $X$ and $Y$, and that the pair $(\ell_{\infty}^n, Y)$ has the \emph{BPBp} for every $n$ if $Y$ is a uniformly convex space.  In the same paper, they asked if the pairs $(c_0, Y)$ and $(\ell_\infty, Y)$ have the BPBp for uniformly convex spaces $Y$. The first author solved the $c_0$ case and proved  \cite{Kim} that  $(c_0, Y)$ have the Bishop-Phelps-Bollob\'as property for all uniformly convex spaces $Y$.

Let $X=L_\infty(\mu)$ or $X=c_0(\Gamma)$ for a set $\Gamma$. Very recently, Lin and authors \cite{Lin} proved that $(X, Y)$ has the BPBp for every uniformly convex space $Y$. So $(L_\infty(\mu), L_p(\nu))$ has the BPBp for all $1<p<\infty$ and for all measures $\nu$. They also proved that $(X, Y)$, as a pair of complex spaces, has the BPBp for every uniformly complex convex space $Y$. In particular, $(L_\infty(\mu), L_1(\nu))$, as a pair of complex spaces, has the BPBp, since $L_1(\nu)$ is uniformly complex convex \cite{Glo}.  

On the other hand, there have been several researches about the BPBp for operators into $C(K)$ spaces (or uniform algebras).
Even though Schachermayer showed \cite{S} that the set of norm attaining operators is not dense in $\mathcal{L}(L_1[0,1], C[0,1])$,  there are some positive results about the BPBp. It is shown \cite{ACK} that $(X, C(K))$ has the BPBp if $X$ is an Asplund space.  This result was extended so that $(X, A)$ has the BPBp if $X$ is Asplund and $A$ is a uniform algebra \cite{CasGuiKad}. The authors also proved \cite{KL} that $(X, C(K))$ has the BPBp if $X^*$ admits a uniformly simultaneously continuous retractions. It is also worthwhile to remark that the pair $(C(K), C(L))$ of the spaces of real-valued continuous functions has the BPBp for every compact Hausdorff spaces $K$ and $L$ \cite{ABCCKLLM}. Concerning the results about $L_\infty$ spaces, it is shown \cite{ACGM} that $(L_1(\mu), L_\infty[0,1])$ has the BPBp and this was generalized \cite{CKLM} so that $(L_1(\mu), L_\infty(\nu))$ has  the BPBp if $\mu$ is any measure and $\nu$ is a localizable measure. These are the strengthening of the results that the set of norm-attaining operators is dense in $\mathcal{L}(L_1(\mu), L_\infty(\nu))$  \cite{FP, PS} for every measure $\mu$ and every localizable measure $\nu$. Finally we remark that if $X$ is uniformly convex, then $(X, Y)$ has the BPBp for every Banach space $Y$ \cite{ABGM, ACKLM, KL2}.

Throughout this paper, we consider only real Banach spaces. It is the main result of this paper that  $(C(K), X)$ has the \emph{BPBp} for every compact Hausdorff space $K$ and for every uniformly convex space $X$. Recall that Schachermayer showed  \cite{S} that every weakly compact operator from $C(K)$ into a Banach space can be approximated by norm attaining weakly compact operators (cf. \cite[Theorem 2]{ACKP}). So the set of all norm attaining operators is dense in $\mathcal{L}(C(K),Y)$ for every reflexive space $Y$.  Notice that the reflexivity of $Y$ is not sufficient to prove that $(C(K), Y)$ has the BPBp. Indeed, if we take a reflexive strictly convex space $Y_0$ which is not uniformly convex, then $(\ell_1^{(2)}, Y_0)$ does not have the BPBp \cite{AAGM2, ACKLM}. If we take $K_0$ as the set consisting of only two points, then $C(K_0)$ is isometrically isomorphic to 2-dimensional $\ell_1^{(2)}$ space. Hence $(C(K_0), Y_0))$ does not have the BPBp. However, if $X$ is uniformly convex, then it will be shown that $(C(K), X)$ has the BPBp.

 \section{Main Result}

Given a Banach space $X$, the modulus of convexity $\delta_X(\epsilon)$ of the unit ball $B_X$ is defined by for $0<\eps<1$,
\[ \delta_X(\epsilon) = \inf \left\{ 1-\left\| \frac{x+y}2 \right\|:  x , y\in B_X, \|x-y\|\ge \epsilon \right\}. \]
A Banach space $X$ is said to be uniformly convex if $\delta_X(\epsilon)>0$ for all $0<\epsilon<1$. It is well known that every uniformly convex space is reflexive.

In \cite{Kim}, the following result was shown:
Let $1>\epsilon>0$ be given and $X$ be a reflexive Banach space and $Y$ be a uniformly convex Banach space with modulus of convexity $\delta_X\left(\epsilon\right)>0$. If $T\in S_{\mathcal{L}\left(X,Y\right)}$ and $x_1\in S_X$ satisfy $$\left\|Tx_1\right\|>1-{{\epsilon}\over{ 2^5}}\delta_X\left({{\epsilon}\over{2}}\right),$$ then there exist $S\in S_{\mathcal{L}\left(X,Y\right)}$ and $x_2\in S_X$ such that $\left\|Sx_2\right\|=1$, $\left\|S-T\right\|<\epsilon$ and $\left\|Tx_1-Sx_2\right\|<\epsilon$.

     This says that for a reflexive space $X$ and  a uniformly convex space $Y$, the pair $(X,Y)$ has a little weaker property than \emph{BPBp}. The only difference from the \emph{BPBp} and the above is approximating the image of a point if the given operator almost attains its norm.  Since the set of all norm attaining operators is dense in $\mathcal{L}(X,Y)$ for every $Y$ if $X$ is reflexive, the following result generalize the result mentioned above \cite{Kim}.

\begin{proposition}\label{prop}
Let $X$ be a Banach space and $Y$ be a uniformly convex space. Suppose that the set of norm attaining operators is dense in $\mathcal{L}(X,Y)$. Then, given $0<\epsilon<1$,  there exists $\eta(\eps)>0$ such that if $T\in S_{\mathcal{L}(X,Y)}$ and $x_1\in S_X$ satisfy $\|Tx_1\|>1-\eta(\eps)$, then there exist $S\in S_{\mathcal{L}(X,Y)}$ and $x_2\in S_X$ such that $\|Sx_2\|=1$, $\|S-T\|<\eps$ and $\|Tx_1-Sx_2\|<\eps$.
\end{proposition}
\begin{proof}
Let $\delta_Y(\cdot)$ be the modulus of convexity of $Y$ and $0<\eps_1<\eps$. Choose $\eps_2>0$ such that $(1-\eps_2^2)^3-2\eps_2-\eps_2^3>1-\delta_Y(\eps_1)$ and $\eps^2_2+2\eps_2+\eps_1<\eps$.

  We show that $\eta(\eps)=\eps_2^2$ is a suitable number. Assume $\|Tx_1\|>1-\eps_2^2$.
Choose $y^*\in S_{Y^*}$ such that $y^*Tx_1=\re y^*Tx_1>1-\eps_2^2$ and define an operator $\tilde{T}_1$ by
$$\tilde{T}_1x=Tx+\eps_2y^*(Tx)Tx_1~\text{for every}~x\in X.$$
It is easy to see that $1-\eps_2<(1-\eps_2^2)(1+\eps_2(1-\eps_2^2))\leq \|\tilde{T}_1x_1\|\leq\|\tilde{T}_1\|\leq 1+\eps_2$.

Let $T_1=\tilde{T}_1/\|\tilde{T}_1\|$. Since the set of norm attaining operators is dense in $\mathcal{L}(X,Y)$, there exist an operator $S$ and $z\in S_X$ such that $\|T_1-S\|<\eps_2^2$ and $\|Sz\|=\|S\|=1$.
Since $\|Sz-T_1z\|<\eps_2^2$, we see that $\|T_1z\|>1-\eps_2^2$, which means that
 $$||Tz+\eps_2y^*(Tz)Tx_1||>(1-\eps_2^2)\|\tilde{T}_1\|>(1-\eps_2^2)(1-\eps_2^2)(1+\eps_2(1-\eps_2^2)).$$
 Hence, we have $|y^*(Tz)|>(1-\eps_2^2)^3-2\eps_2-\eps_2^3>1-\delta_Y(\eps_1)$. Choose $\alpha=\pm 1$ satisfying $y^*T(\alpha z) = |y^*T(z)|$ and  let $x_2 = \alpha z$. Then
 \[\left\|  \frac{Tx_1 + Tx_2 }{2}\right\| \ge   \frac{y^*Tx_1 + y^*Tx_2 }{2}>1-\delta_Y(\epsilon_1).\]
Hence, we see that $\|Tx_1-Tx_2\|<\eps_1$.
Moreover,
  \begin{eqnarray*}
\|Sx_2-Tx_1\|
&\leq &\|Sx_2-T_1x_2\|+\|T_1x_2-\tilde{T}_1x_2\|+\|\tilde{T}_1x_2-Tx_2\|+\|Tx_2-Tx_1\|\\
&\leq&\|S-T_1\|+|\|\tilde{T}_1\|-1|+\eps_2+\eps_1\\
&<& \eps^2_2+\eps_2+\eps_2+\eps_1<\eps.
   \end{eqnarray*} This completes the proof.
\end{proof}

Now we state the main theorem of this paper.

\begin{theorem}\label{thm}
Let $X$ be a uniformly convex space and $K$ be a compact Hausdorff space. Then the pair $(C(K), X)$ has the BPBp.
\end{theorem}

 Before we present the proof of the main result, we begin with preliminary comments on vector measure and two lemmas.
Recall that a vector measure $G:\Sigma\to X$ on a $\sigma$-algebra $\Sigma$ is said to be countably additive if, for every mutually disjoint sequence of $\Sigma$-measurable subsets $\{A_i\}_{i=1}^\infty$, we have
\[ G\left(\bigcup_{i=1}^\infty A_i\right) = \sum_{i=1}^\infty G(A_i).\]
For a $\Sigma$-measurable subset $A$, the semi-variation $\|G\|(A)$ of $G$  is defined by
\[ \|G\|(A) = \sup\{ |x^*G|(A) : x^*\in B_{X^*}\},\] where $|x^*G|(A)$ is the total variation of the scalar-valued countably additive measure $x^*G$ on $A$. The vector measure $G$ on a Borel $\sigma$-algebra is said to be regular if for each Borel subset $E$ and $\epsilon>0$ there exists a compact subset $K$ and an open set $O$ such that
\[ K\subset E \subset O \ \ \ \ \text{and}\ \ \ \ \|G\|(O\setminus K)<\epsilon.\]
It is well known that if $X$ is reflexive, each operator $T$ in $\mathcal{L}(C(K),X)$ has a $X$-valued countably additive representing Borel measure $G$ and the measure is regular (see \cite[VI. Theorem 1, 5 and Corollary 14]{DU} for a reference). That is, for all $f\in C(K)$ and $x^*\in X^*$, we have
\[ Tf = \int_K f\, dG, \ \  \ x^*T(f) = \int_K f \, d^*G \ \ \ \text{and} \ \ \  \|T\| = \|G\|(K).\]
If $G$ is a countably additive representing measure for an operator $T$ in   $\mathcal{L}(C(K),X)$, then it is easy to see that for any bounded Borel measurable function  $h:K \to \mathbb{R}$, the mapping $S$, defined by $Sf=\int f hdG$, is a bounded linear operator and $\|S\|\leq\|T\|\cdot \|h\|_\infty$, where $\|h\|_\infty = \sup\{ |h(k)|: k\in K\}$.

\begin{lemma}\label{uniform1} Let  $G$ be a countably additive, Borel regular $X$-valued vector measure on a compact Hausdorff space $K$ with $\|G\|(K)=1$ and let  $0<\eta,\gamma<1$. Assume that $f\in S_{C(K)}$ and $x^*\in S_{X^*}$ satisfy  $$\int_K f \, dx^*G>1-\eta.$$
Then, we have \[|x^*G|\big(K\setminus \left(A^+_{\gamma}\cup A^-_{\gamma}\right)\big)<2{{\eta}\over{\gamma}}+\eta,\] where
$A^+_{\gamma}=\{t\in K~|~f(t)\geq1-\gamma\}$ and  $A^-_{\gamma}=\{t\in K~|~f(t)\leq-1+\gamma\}$. Moreover, there exist mutually disjoint compact sets $F^+, F^-$ such that $x^*G$ is positive on $F^+$, negative on $F^-$ and $$\int_{(F^+\cap A^+_{\gamma})\cup(F^-\cap A^-_{\gamma})} f\, dx^*G>1-4{{\eta}\over{\gamma}}.$$
\end{lemma}
\begin{proof}
The Hahn decomposition of $x^*G$ and the regularity of $G$ show that there exist mutually disjoint compact sets $F^+$, $F^-$ such that  $x^*G$ is positive on $F^+$, negative on $F^-$ and $\|G\|\big(K\setminus (F^+\cup F^-)\big)<\eta$.
\begin{eqnarray*}
 1-\eta
 &\leq&\int_K f\, d x^*G = \int_{F^+} f d x^*G+ \int_{F^-} f\,d x^*G+\int_{K\setminus  (F^+\cup F^-)}f\,dx^*G\\
 &=& \int_{F^+\cap A^+_{\gamma}} f\, d x^*G+ \int_{F^+\setminus A^+_{\gamma}}f\, d x^*G+\int_{F^-\cap A^-_{\gamma}} f \,d x^*G+ \int_{F^-\setminus A^-_{\gamma}}f\, d x^*G+\int_{K\setminus  (F^+\cup F^-)}f\,dx^*G\\
 &\leq& x^*G(F^+\cap A^+_{\gamma})+(1-\gamma)x^*G(F^+\setminus A^+_{\gamma})-x^*G(F^-\cap A^-_{\gamma})-(1-\gamma)x^*G(F^-\setminus A^-_{\gamma})+\eta\\
 &=&x^*G(F^+)-x^*G(F^-)-\gamma(x^*G(F^+\setminus A^+_{\gamma})-x^*G(F^-\setminus A^-_{\gamma}))+\eta.
 \end{eqnarray*}
 Since
 $ x^*G(F^+)-x^*G(F^-)= |x^*G|(F^+\cup F^-) \le \|G\|(K)= 1$,
 we get
 $$|x^*G|((F^+\setminus A^+_{\gamma}) \cup (F^-\setminus A^-_{\gamma}))= x^*G(F^+\setminus A^+_{\gamma})-x^*G(F^-\setminus A^-_{\gamma})\leq 2{{\eta}\over{\gamma}}.$$
 This shows that
\begin{align*}
|x^*G|\big(K\setminus \left(A^+_{\gamma}\cup A^-_{\gamma}\right)\big)  & \leq  |x^*G|\big(K\setminus (F^+\cup F^-)) + |x^*G|\big(F^+\cup F^-) \setminus  \left(A^+_{\gamma}\cup A^-_{\gamma}\right)\big)\\
&\leq  \|G\|\big(K\setminus (F^+\cup F^-)) + |x^*G|\big( (F^+\setminus A^+) \cup ( F^-\setminus  A^-_{\gamma})\big)\\
&<2{{\eta}\over{\gamma}}+\eta
\end{align*}
 and
\begin{align*}
\int_{(F^+\cap A^+_{\gamma})\cup(F^-\cap A^-_{\gamma})} f\, dx^*G &= \int_{F^+\cup F^-} f\, dx^*G - \int_{(F^+\setminus A^+_{\gamma})\cup(F^-\setminus A^-_{\gamma})} f\, dx^*G
\\
&\geq \int_{K} f\, dx^*G - \|G\|(K\setminus (F^+\cup F^-)) - |x^*G|((F^+\setminus A^+_{\gamma}) \cup (F^-\setminus A^-_{\gamma}))\\
&>1-2\eta-2{{\eta}\over{\gamma}}>1-4{{\eta}\over{\gamma}}.
\end{align*} This completes the proof.
\end{proof}

\begin{lemma}\label{uniform2}Let $X$ be a uniformly convex space with the modulus of convexity $\delta_X$ and $T\in S_{\mathcal{L}(C(K),X)}$ be an operator represented by the countably additive, Borel regular vector measure $G$.  Let $0<\epsilon<1$ and $A$ be a Borel set of $K$.  Suppose that an operator $S$, defined by $Sf=\int_A fdG$, satisfies $\|S\|>1-\delta_X(\epsilon).$ Then $$\|T-S\|=\sup_{f\in B_{C(K)}}\left\|\int_{K\setminus A}fdG\right\|<\epsilon.$$
\end{lemma}
\begin{proof}
Choose $x^*\in S_{X^*}$, $f_0\in S_{C(K)}$ such that $\|Sf_0\|=x^*Sf_0 > 1-\delta_X(\eps)$.
By the regularity of $G$, we may choose a compact set $A_1\subset A$ such that $$\int_{A_1} f_0dx^*G >1-\delta_X(\epsilon).$$
Fix a closed set $B \subset{K\setminus A}$ and $g\in B_{C(B)}$. Then, choose $g_+$, $g_-\in B_{C(K)}$ satisfying
\begin{align*}
    g_+(t)&=g_-(t)=f_0(t)\ \ \text{for}\  t\in A_1 \ \ \text{and}\\
   g_+(t)&=- g_-(t)=g(t)\ \  \text{for}\ t\in B.
\end{align*}
So, we have
$$1-\delta_X(\epsilon)<\int_{A_1} f_0 dx^*G\leq\left\|\int_{A_1} f_0 dG\right\|=\frac 12\left\|{{\int_{A_1\cup B}g_+ dG+ \int_{A_1 \cup B}g_- dG}}\right\|.$$
Note that $\left\|\int_{A_1\cup B}g_+ dG\right\|,\left\|\int_{A_1\cup B}g_+ dG\right\|\leq 1$. Thus, from the uniform convexity of $X$, we get that $$\left\|2\int_{B}g dG\right\|=\left\|\int_{A_1 \cup B}g_+ dG-\int_{A_1 \cup B}g_- dG\right\|<\epsilon.$$
This implies $\|T-S\|<\epsilon$ and the proof is done.
\end{proof}

\begin{proof}[Proof of Theorem~\ref{thm}]
Let $\delta_X$ be the modulus of convexity for $B_X$. Fix $0<\epsilon<{{1}\over{2^8}}$ and let $\eta$ be the function which appears in Proposition \ref{prop} for the pair $(C(K), X)$, and let $\gamma(t)=\min\left\{\eta(t),\delta_X(t),{{t}\over{3}}\right\}$ for $t \in (0,1)$. Assume that $T\in S_{\mathcal{L}(C(K), X)}$ and $f_0\in S_{C(K)}$ satisfy that $$\|Tf_0\|>1-{{\epsilon}\over{8}}\gamma\left({{\epsilon}\over{6}}\delta_X\left({{\epsilon}\over{6}}\right)\right).$$
Let $G$ be the representing vector measure for $T$ which is countably additive Borel regular on $K$.
 Choose $x_1^*\in S_{X^*}$ such that $x_1^*Tf_0>1-{{\epsilon}\over{8}}\gamma\left({{\epsilon}\over{6}}\delta_X\left({{\epsilon}\over{6}}\right)\right)$.  Bt Lemma \ref{uniform1} there exist two mutually disjoint compact sets $F^+, F^-$ such that $x^*G$ is positive on $F^+$, negative on $F^-$ and $$\int_{(F^+\cap A^+_{\epsilon/2})\cup(F^-\cap A^-_{\epsilon/2})} f \, dx^*G>1-\gamma\left({{\epsilon}\over{6}}\delta_X\left({{\epsilon}\over{6}}\right)\right),$$ where $A^+_{\epsilon/2}=\{t\in K~|~f_0(t)\geq1-\frac\epsilon2\}$ and $A^-_{\epsilon/2}=\{t\in K~|~f_0(t)\leq-1+\frac\epsilon2\}$.

Let $A_1 = F^+\cap A^+_{\epsilon/2}$, $A_2= F^-\cap A^-_{\epsilon/2}$ and $A=A_1\cup A_2$. Then, define $S_1\in B_{\mathcal{L}(C(K), X)}$ by $S_1f=\int_A f dG$ for every $f\in C(K)$.
Then Lemma \ref{uniform2} shows that $\|T-S_1\|<{{\epsilon}\over{6}}$.
Choose $f_1\in S_{C(K)}$ such that
\begin{align*}
    f_1(t)&= 1\ \ \ \ \text{for}\  \ t\in A_1\ \ \ \  \text{and}\\
   f_1(t)&=-1 \ \ \text{for}\ \ t\in A_2.
\end{align*}
For $f\in C(K)$, the restriction of $f$ to $A$ will be denoted by $f|_A$. Now consider $S_1$ as an operator in $\mathcal{L}(C(A),X)$. Then we have
$$\|S_1(f_1|_A)\|>1-\gamma\left({{\epsilon}\over{6}}\delta_X\left({{\epsilon}\over{6}}\right)\right),$$
So Proposition~\ref{prop}  shows that   there exist $S_2\in S_{\mathcal{L}(C(A),X)}$ and $f_2\in S_{C(A)}$ such that $\|S_2f_2\|=1$, $\left\|S_2-{{S_1}\over{\|S_1\|}}\right\|<{{\epsilon}\over{6}}\delta_X\left({{\epsilon}\over{6}}\right)$ and $\left\|S_2f_2-{{S_1(f_1|_A)}\over{\|S_1\|}}\right\|<{{\epsilon}\over{6}}\delta_X\left({{\epsilon}\over{6}}\right)$.
Let $G'$ be the representing vector measure for $S_2$ which is  countably additive Borel regular on $A$. Choose $x^*_2\in S_{X^*}$ so that $x^*_2S_2f_2=\|S_2f_2\|=\int_{A}f_2\, dx^*_2G'=1$.

Since
\begin{align*}
x^*_2S_2(f_1|_A+f_2)
&\geq 2x^*_2S_2f_2-\left\|S_2f_2-S_2(f_1|_A)\right\|\\
&\geq 2-\left\|S_2f_2-{{S_1(f_1|_A)}\over{\|S_1\|}}\right\|-\left\|{{S_1(f_1|_A)}\over{\|S_1\|}}-S_2(f_1|_A)\right\|\\
&>2\left(1-{{\epsilon}\over{6}}\delta_X\left({{\epsilon}\over{6}}\right)\right),
\end{align*}
we get
\[\int_{A}{{f_1+f_2}\over{2}}dx^*G'>1-{{\epsilon}\over{6}}\delta_X\left({{\epsilon}\over{6}}\right).\] By applying Lemma~\ref{uniform1} again, we get a compact subset $F$ of $A$ such that
\[ F \subset \left\{t\in A : \left| {f_1(t)+f_2(t)}\right|> 2(1-\epsilon) \right\} \] and
 $$\left\|\int_{F}{{f_1+f_2}\over{2}}\, dG'\right\|>1-\delta_X\left({{\epsilon}\over{6}}\right).$$
Let $B= \left\{t\in A :  f_1(t)f_2(t)\ge 0 \right\}$. Then,  $F\subset B$ and
 $$ \sup_{f\in B_{C(A)} } \left\|\int_B f \, dG'  \right\| \geq\left\|\int_{F}{{f_1+f_2}\over{2}}\, dG'\right\|>1-\delta_X\left({{\epsilon}\over{6}}\right).$$ By Lemma~\ref{uniform2}, we have
\[\sup_{f\in B_{C(K)}}\left\|\int_{A\setminus B}f\, dG'\right\|<{{\epsilon}\over{6}}.\]
Define $S\in {\mathcal{L}(C(A),X)}$ by, for $f\in C(A)$,
\[Sf=\int_{B}f\, dG'-\int_{A\setminus B}f\, dG'\]  and let
\[
f_3 = \left\{
  \begin{array}{ll}
   \ |f_2|& \text{for}\  t\in A_1,\\
   -|f_2|& \text{for}\  t\in A_2.
  \end{array}
\right.
\] So $f_3\in C(A)$ and $f_3 = f_2 \chi_{B} - f_2 \chi_{A\setminus B}$, where $\chi_S$ is the characteristic function on a set $S$. Hence we have  $Sf_3=S_2f_2$,  $\|Sf_3\|=\|S\|=1$ and $\|S-S_2\|<{{\epsilon}\over{3}}$. On the other hand, we have $\|2f_3-f_1|_A\|\leq 1$. Since $X$ is uniformly convex and we have
$Sf_3={{S(f_1|_A)+S(2f_3-f_1|_A)}\over{2}}$, we get
\[Sf_3=S(f_1|_A)=S(2f_3-f_1|_A).\]

We now consider $S_1, S_2, S$ as operators in $\mathcal{L}(C(K),X)$ using the canonical extension. That is, $S(f) = S(f|_A)$, $S_i(f)=S_i(f|_A)$ for all $f\in C(K)$ and for $i=1,2$. Let $C$ be the compact subset defined by $$C=\{t\in K : |f_1(t)-f_0(t)|\ge \epsilon\}.$$   Note that $A$ and $C$ are mutually disjoint. Indeed, if $t\in A$, then $|f_0(t) - f_1(t)|\le \epsilon/2$. So there is $\phi \in C(K)$ such that $0\leq \phi \leq 1$, $\phi(k)=1$ for $k\in A$ and $\phi(k)=0$ for $k\in C$. Let $g=\phi f_1+(1-\phi)f_0$. Then we see that $\|Sg\|=1$,
  \begin{eqnarray*}
  \|S-T\|
  &\leq&\|S-S_2\|+\|S_2-{{S_1}\over{\|S_1\|}}\|+\|{{S_1}\over{\|S_1\|}}-S_1\|+\|S_1-T\|\\
  &<&{{\epsilon}\over{3}}+{{\epsilon}\over{6}}+{{\epsilon}\over{3}}+{{\epsilon}\over{6}}=\epsilon
  \end{eqnarray*}
and $\|g-f_0\|= \sup_{k\in K\setminus C}|\phi(k)(f_1(k)-f_0(k))|< \epsilon$. This completes the proof.
\end{proof}


\begin{thebibliography}{99}

\bibitem{AAGM2} M.~D.~Acosta, R.~M.~ Aron, D.~Garc\'ia and M.~Maestre, \emph{The Bishop-Phelps-Bollob\'as Theorem for operators}, J. Funct. Anal. {\bf 254} (2008), 2780-2799.


\bibitem{ABCCKLLM}
M.~D.~Acosta, J.~Becerra-Guerrero, Y.~S.~Choi, M.~Ciesielski, S.~K.~Kim, H.~J.~Lee, M.~L.~Lorenc\c{o} and M.~Mart\'{\i}n, \emph{The Bishop-Phelps-Bollob\'{a}s property for operators between spaces of continuous functions}, Nonlinear Anal. {\bf 95}~ (2014), 323-332.


\bibitem{ABGM}
M.~D.~Acosta, J.~Becerra-Guerrero, D.~Garc\'{\i}a and M.~Maestre, \emph{The Bishop-Phelps-Bollob\'{a}s Theorem for bilinear forms},  Trans. Amer. Math. Soc. {\bf 365} (2013) 5911-5932.

\bibitem{ACK} R.~M.~Aron, B.~Cascales and O.~Kozhushkina, \emph{The Bishop-Phelps-Bollob\'as Theorem and Asplund operators}, Proc. Amer. Math. Soc. {\bf 139} (2011), 3553-3560.

\bibitem{ACKLM} R.~M.~Aron, Y.~S.~Choi, S.~K.~Kim, H.~J.~Lee and M. Mart\'in, \emph{The Bishop-Phelps-Bollob\'{a}s version of Lindenstrauss properties A and B}, Preprint.


\bibitem{ACKP}  J.~Alaminos, Y.~S.~Choi, S.~G.~Kim and R.~Pay\'a, \emph{Norm attaining bilinear forms on spaces of continuous functions}, Glasgow Math. J. {\bf 40} (1998), 467-482.


\bibitem{ACGM} R.~M.~Aron, Y.~S.~Choi, D.~Garc\'ia and M.~Maestre, \emph{The Bishop-Phelps-Bollob\'as Theorem for $L(L_1(\mu),L_\infty[0,1])$}, Adv. Math. {\bf 228} (2011) 617-628.

\bibitem{BP} E.~Bishop and R.~R.~Phelps, \emph{A proof that every Banach space is subreflexive.}, Bull. Amer. Math. Soc. {\bf 67} (1961), 97-98.

\bibitem{Bol} B.~Bollob\'as, \emph{An extension to the theorem of Bishop and Phelps}, Bull. London. Math. Soc. {\bf 2} (1970), 181-182.

\bibitem{Bou} J.~Bourgain, \emph{Dentability and the Bishop-Phelps property}, Israel J. Math. {\bf 28} (1977), 265-271.


\bibitem{CasGuiKad}
B.~Cascales, A.~J.~Guirao and V.~Kadets, \emph{A Bishop-Phelps-Bollob\'{a}s type theorem for uniform algebras}, {Adv. Math.} \textbf{240}~(2013), 370--382.

\bibitem{CKlush} Y.~S.~Choi and S.~K.~Kim, \emph{The Bishop-Phelps-Bollob\'as property and lush spaces}, J. Math. Anal. Appl. {\bf 390}~(2013) 549-555.


\bibitem{CK2} Y.~S.~Choi and S.~K.~Kim, \emph{The Bishop-Phelps-Bollob\'as theorem for operators from $L_1(\mu)$ to Banach spaces with the Radon-Nikod\'ym property}, J. Funct. Anal. {\bf 261}~(2013), 1446-1456.

\bibitem{CKLM} Y.~S.~Choi, S.~K.~Kim, H.~J.~Lee and M. Mart\'in, \emph{The Bishop-Phelps-Bollob\'{a}s theorem for operators on $\boldsymbol{L_1(\mu)}$}, J. Funct. Anal. {\bf 267}~(2014), no. 1, 214-242.

\bibitem{CLS1}
Y.~S.~Choi, H.~J.~Lee and H.~G.~Song, \emph{Denseness of norm-attaining mappings on {B}anach spaces}, Publ. Res. Inst. Math. Sci.
 {\bf 46} (2010), 171-182.

\bibitem{DU} J.~Diestel and J.~J.~Uhl, Jr, \emph{Vector Measures}, Amer. Math. Soc., Math. Surveys. {\bf 15}, 1977.

\bibitem{FP} C.~Finet and R.~Pay\'{a}, \emph{Norm attaining operators from $L_1$ into $L_{\infty}$}, Israel J. Math. {\bf 108} (1998), 139-143.


\bibitem{Glo} \textsc{J.~Globevnik}, {On complex strict and uniform convexity,} \emph{Proc. Amer. Math. Soc.}  {\bf 47}~(1975), 175-178.


\bibitem{JoWo}
J.~Johnson and J.~Wolfe, \emph{Norm attaining operators}, Studia Math. {\bf 65} (1998), 7--19.


\bibitem{Kim}
S.~K.~Kim, {\it The Bishop-Phelps-Bollob\'as Theorem for operators from $c_0$ to
uniformly convex spaces}, Israel J. Math. {\bf 197} (2013), 425-435.

\bibitem{KL}
S.~K.~Kim and H.~J.~Lee, \textit{Simultaneously continuous retraction and Bishop-Phelps-Bollob\'as type theorem}, to appear in J. Math. Anal. Appl. (DOI: 10.1016/j.jmaa.2014.06.009)

\bibitem{KL2} S.~K.~Kim and H.~J.~Lee, \emph{Uniform Convexity and Bishop-Phelps-Bollob\'as Property}. Canad. J. Math. {\bf 66}~(2014), no. 2, 373-386.

\bibitem{Lin}
S.~K.~Kim, H.~J.~Lee and P.~K.~Lin,  \textit{The Bishop-Phelps-Bollob\'as theorem for operators from $L_\infty(\mu)$ to uniformly convex spaces}. Preprint.


\bibitem{L} J.~Lindenstrauss, \emph{On operators which attain their norm}, Israel J. Math. {\bf 1} (1963) 139-148.

\bibitem{PS} R.~Pay\'{a} and Y.~Saleh, \emph{Norm attaining operators from $L_1(\mu)$ into $L_{\infty}(\nu)$}, Arch. Math. {\bf 75} (2000), 380-388.


\bibitem{S} W.~Schachermayer, \emph{Norm attaining operators on some classical Banach spaces}, Pacific J. Math. {\bf 105} (1983), 427-438.

\bibitem{U} J.~J.~Uhl, Jr, \emph{Norm attaining operators on $L_1[0,1]$ and the Radon-Nykod\'ym property}, Pacific J. Math. {\bf 63} (1976), 293-300.

\end{thebibliography}
\end{document}